\documentclass[graybox]{svmult}

\usepackage{mathptmx}       
\usepackage{helvet}         
\usepackage{courier}        
\usepackage{type1cm}        
\usepackage{makeidx}         
\usepackage{graphicx}        
\usepackage{multicol}        
\usepackage[bottom]{footmisc}

\makeindex             

\usepackage{amsmath}
\usepackage{amssymb}
\usepackage{verbatim}
\usepackage{graphicx}
\usepackage{xcolor}
\usepackage{fancybox}
\usepackage[square,numbers]{natbib}
\usepackage{enumitem}
\usepackage{eurosym}

\usepackage{amsmath,amssymb}
\usepackage{amscd}
\usepackage{color}
\usepackage{array, multirow}
\usepackage{mathcomp}

\usepackage{graphicx}
\usepackage{bbm}
\usepackage{enumitem}
\usepackage{enumerate}
\usepackage{mathtools}

\usepackage{pdfsync}%
\definecolor{darkblue}{rgb}{0.0,0.0,0.7}
\RequirePackage[%
colorlinks = true,%
linkcolor = darkblue,%
citecolor = darkblue,%
urlcolor = darkblue, %
]{hyperref}%
\hypersetup{%
  pdfauthor = {Bellec, Lecu\'e, Tsybakov},%
  pdfcreator = {pdflatex},%
  pdfproducer = {pdflatex},}

\usepackage{cleveref}

\usepackage{thmtools}

\DeclareMathOperator*{\argmin}{argmin}
\DeclareMathOperator*{\supp}{Supp}

\def\ds1{\textrm{1\kern-0.25emI}} 

\newcommand\R{{\mathbb R}}
\renewcommand\E{{\mathbb E}}

\newcommand{\Pro}{\mathbb{P}}

\usepackage{marginnote}

\newcommand{\vf}{{\mathbf{f}}}
\newcommand{\vtheta}{{\boldsymbol{\theta}}}
\newcommand{\vbeta}{{\boldsymbol{\beta}}}

\newcommand{\vv}{{\mathbf{v}}}
\newcommand{\vh}{{\mathbf{h}}}
\newcommand{\ve}{{\boldsymbol{e}}}
\newcommand{\vzero}{{\mathbf{0}}}
\newcommand{\vx}{{\mathbf{x}}}
\newcommand{\vz}{{\mathbf{z}}}

\newcommand{\vu}{{\mathbf{u}}}
\newcommand{\vw}{{\boldsymbol{w}}}

\newcommand{\hmu}{\boldsymbol{\hat \mu}}

\newcommand{\vy}{\mathbf{y}}
\newcommand{\vxi}{{\boldsymbol{\xi}}}

\newcommand{\vdelta}{{\boldsymbol{\Delta}}}

\newcommand{\design}{\mathbb{X}}

\newcommand{\hbeta}{{\boldsymbol{\hat\beta}}}

\newcommand{\norma}[1]{\vert#1\vert}
\newcommand{\norm}[1]{\Vert#1\Vert}

\DeclareMathOperator{\Med}{Med}

\begin{document}

\title*{Bounds on the prediction error of penalized least squares estimators with convex penalty}
\titlerunning{Prediction error of convex penalized least squares estimators}
\author{Pierre C. Bellec and Alexandre B. Tsybakov}
\institute{
    Pierre C. Bellec \at ENSAE, 3 avenue Pierre Larousse, 92245 Malakoff Cedex, France    \email{pierre.bellec@ensae.fr}
    \and
    Alexandre B. Tsybakov \at ENSAE, 3 avenue Pierre Larousse, 92245 Malakoff Cedex, France \email{alexandre.tsybakov@ensae.fr}\
}
%
%
\maketitle

\abstract{
This paper considers the penalized least squares estimator with arbitrary convex penalty. When the observation noise is Gaussian, we show that the prediction error is a subgaussian random variable concentrated around its median. We apply this concentration property to derive sharp oracle inequalities for the prediction error of the LASSO, the group LASSO and the SLOPE estimators, both in probability and in expectation.   
In contrast to the previous work on the LASSO type methods, our oracle inequalities in probability are obtained at any confidence level for estimators with tuning parameters that do not depend on the confidence level. This is also the reason why we are able to establish sparsity oracle bounds in expectation for the LASSO type estimators, while the previously known techniques did not allow for the control of the expected risk. 
In addition, we show that the concentration rate in the oracle inequalities  
is better than it was commonly understood before.
}

\section{Introduction and notation}

Assume that we have noisy observations
\begin{equation}\label{1}
    y_i = f_i + \xi_i,
    \qquad
    i = 1,...,n,
\end{equation}
where $\xi_1,...,\xi_n$ are i.i.d. centered Gaussian random variables with variance $\sigma^2$, and $\vf=(f_1,...,f_n)^T\in\R^n$ is an unknown mean vector.
In vector form, the model \eqref{1} can be rewritten as $\vy = \vf + \vxi$ where $\vxi=(\xi_1,...,\xi_n)^T$ and $\vy = (y_1,...,y_n)^T$.
Let $\design\in\R^{n\times p}$ be a matrix with $p$ columns that we will call the design matrix. We consider the problem of estimation of $\vf$ by $\design \hbeta(\vy)$ where  $\hbeta(\vy)$ is an estimator valued in $\R^p$. Specifically, we restrict our attention to penalized least squares estimators of the form
\begin{equation}
    \hbeta(\vy) \in \argmin_{\vbeta\in\R^p} \left(\norm{\vy - \design\vbeta}^2 + 2 F(\vbeta)\right),
    \label{eq:def-hbeta-pen}
\end{equation}
where $\norm{\cdot}$ is the scaled Euclidean norm defined by
$\norm{\vu}^2 = \frac 1 n \sum_{i=1}^n u_i^2$ for any $\vu=(u_1,...,u_n)$, and $F:\R^p\rightarrow[0,+\infty]$ is a proper convex function, that is, a nonnegative convex function such that $F(\vbeta)<+\infty$ for at least one $\vbeta\in\R^p$. If the context prevents any ambiguity, we will omit the dependence on $\vy$ and
write $\hbeta$ for the estimator $\hbeta(\vy)$.

The object of study in this paper is  the prediction error of the estimator $\hbeta(\vy)$, that is, the value $\norm{\design\hbeta(\vy) - \vf}$. We show that, for any design matrix $\design$ and for any proper convex penalty $F$, the prediction error $\norm{\design\hbeta(\vy) - \vf}$ is a subgaussian random variable concentrated around its median and its expectation. 
Furthermore, when $F$ is a norm, we obtain an explicit formula for the predictor $\design \hbeta(\vy)$ in terms of the projection on the dual ball.  Finally, we apply the subgaussian concentration property around the median to derive sharp oracle inequalities for the prediction error of the LASSO, the group LASSO and the SLOPE estimators, both in probability and in expectation.   
The inequalities in probability improve upon the previous work on the LASSO type estimators (see, e.g., the papers \cite{bickel2009simultaneous,koltchinskii2011nuclear,lptv2011,dalalyan2014prediction}
or the monographs \cite{giraud2014introduction,buhlmann2011statistics,vdgeer}) since, in contrast to the results of these works, our bounds hold at any given confidence level for estimators with tuning parameter that does not depend on the confidence level. This is also the reason why we are able to establish bounds in expectation, while the previously known techniques did not allow for the control of the expected risk. In addition, we show that the concentration rate in the oracle inequalities  
is better than it was commonly understood before.
Similar results have been obtained quite recently in \cite{bellec2016slope} by a different and somewhat more involved construction conceived specifically for the LASSO and the SLOPE estimators. The techniques of the present paper are more general since they can be used not only for these two estimators but for any penalized least squares estimators with convex penalty.

\section{Notation}

For any random variable $Z$, let $\Med[Z]$ be a median of $Z$, i.e.,
any real number $M$ such that $\mathbb P(Z\ge M) = \mathbb P (Z\le M) = 1/2$.
For  a vector 
$\vu=(u_1,...,u_n)$, the sup-norm, the Euclidean norm and the $\ell_1$-norm will be denoted 
by
$\norma{\vu}_\infty = \max_{i=1,...,n} |u_i|$,
$\norma{\vu}_1 = \sum_{i=1}^n |u_i|$
and
$
\norma{\vu}_2 = ( 
    \sum_{i=1}^n u_i^2
)^{1/2}
$. The inner product in $\R^n$ is denoted by $\langle\cdot,\cdot\rangle$. We also denote by  $\supp(\vu)$ the support of $\vu$, and by $|\vu|_0$ the cardinality of $\supp(\vu)$. We denote by $I(\cdot)$ the indicator function.
For any $S\subset\{1,...,p\}$ and a vector $\vu=(u_1,...,u_p)$, we set $\vu_S=(u_jI(j\in S))_{j=1,...,p}$, and we denote by $|S|$ the cardinality of~$S$.

\section{The prediction error of convex penalized estimators is subgaussian}
\label{s:machinery}

The aim of this section is 
to show that the prediction error $\norm{\design\hbeta(\vy) - \vf}$
is subgaussian and concentrates around its median 
for any estimator $\hbeta(\vy)$ defined by \eqref{eq:def-hbeta-pen}.
The results of the present section will allow us to carry out a unified analysis of LASSO, group LASSO  and SLOPE estimators 
in Sections \ref{s:lasso} -- \ref{s:slope}. 

\begin{proposition}\label{prop1}
Let $\hmu:\R^n\rightarrow\R^n$ be an $L$-Lipschitz function,
that is, a function satisfying
\begin{equation}
    \label{eq:lipschitz-hmu}
    \norm{\hmu(\vy) - \hmu(\vy')} \le L \norm{\vy - \vy'},
    \qquad \forall
    \vy,\vy'\in\R^n.
\end{equation}
    Let $f(\vz) = \norm{\hmu(\vf+\sigma \vz) - \vf}$
    for some fixed $\vf\in\R^n$ and
    $\vz\sim\mathcal N(\vzero,I_{n\times n})$.
    Then,
     for all $t>0$, 
    \begin{equation}
        \label{eq:lipschitz-concentration}
            \Pro
            \left(
                f(\vz)>\Med[f(\vz)] + \frac{\sigma L t}{ \sqrt{n}}
            \right) \le 1-\Phi(t),
                       \end{equation}
where $\Phi(\cdot)$ is the $\mathcal N(0,1)$ cumulative distribution function.
\end{proposition}
\begin{proof}
    \smartqed
    The result  
    follows immediately from the Gaussian concentration inequality  (cf., e.g.,  \cite[Theorem 6.2]{lifshits})
    and the fact that $f(\cdot)$ satisfies the Lipschitz condition
    \begin{equation*}
        \left|
            f(\vu)
            -
            f(\vu')
        \right|
        \le
        \norm{\hmu(\vf+\sigma \vu) - \hmu(\vf+\sigma \vu')}
        \le \frac{\sigma L}{\sqrt n} \norma{\vu - \vu'}_2, \qquad \forall
    \vu,\vu'\in\R^n.
    \end{equation*}
        \qed
\end{proof}

We now show  that $\hmu(\vy)=\design\hbeta(\vy)$ where $\hbeta(\vy)$ is  estimator \eqref{eq:def-hbeta-pen} with any proper convex penalty satisfies the Lipschitz condition \eqref{eq:lipschitz-hmu} with $L=1$.

We first consider estimators penalized by a norm in $\R^p$, for which we get a sharper result. Namely, in this case the explicit expression for $\design\hbeta(\vy)$ is available. In addition, we get a stronger property than the Lipschitz condition \eqref{eq:lipschitz-hmu}.
Let $N:\R^p\rightarrow\R_+$ be a norm and let $N_\circ(\cdot)$
be the corresponding dual norm defined by $N_\circ(\vu) = \sup_{\vv\in\R^p: N(\vv) = 1} \vu^T\vv$.
For any $\vy\in\R^n$, define $\hbeta(\vy)$ as a solution of the following minimization problem:
\begin{equation}
    \hbeta(\vy) \in \argmin_{\vbeta\in\R^p} \left(\norm{\vy - \design\vbeta}^2 + 2 N(\vbeta)\right).
    \label{eq:def-hbeta-norm-N}
\end{equation}
The next two propositions are crucial  in proving that the concentration bounds \eqref{eq:lipschitz-concentration} apply when $f(\vz)$ is the prediction error associated with $\hbeta(\vy)$. 

\begin{proposition}\label{prop2}
Let $N:\R^p\rightarrow\R_+$ be a norm and let  $\hbeta(\vy)$ be a solution of \eqref{eq:def-hbeta-norm-N}.
    For all $\vy\in\R^n$ and all matrices $\design\in\R^{n\times p}$ , we have:\\
    (i)  
    $\design\hbeta(\vy) = \vy-P_C(\vy)$
    where  $P_C(\cdot):\R^n\rightarrow C$ 
    is the operator of projection onto the closed convex set
    $C = \{\vu\in\R^n: N_\circ(\design^T\vu) \le 1/n \}$, \\
    (ii) the function $\hmu(\vy) = \design\hbeta(\vy)$
    satisfies 
    $$
    \norm{\hmu(\vy) - \hmu(\vy')}^2 \le \norm{\vy - \vy'}^2 - \frac1n |P_C(\vy)- P_C(\vy')|_2^2.
    $$
\end{proposition}

\begin{proof}
    \smartqed
    Since $C$ is a closed convex set, we have that $\vtheta = P_C(\vy)$ if and only if 
    \begin{equation}\label{obtu}
        (\vy - \vtheta)^T(\vtheta - \vu) \ge 0
        \qquad
        \text{for all }
        \vu\in C.
    \end{equation}
 Thus, to prove statement (i) of the proposition, it is enough to check that \eqref{obtu} holds for $\vtheta = \vy - \design\hbeta(\vy)$.
  Since \eqref{obtu} is trivial when $\hbeta(\vy)=0$, we assume in what follows that $\hbeta(\vy)\ne0$. Any solution $\hbeta(\vy)$ of the convex minimization problem in \eqref{eq:def-hbeta-norm-N} satisfies
    \begin{equation}\label{moreau}
        \frac1{n}\design^T(\design\hbeta(\vy) - \vy) + \vv =\vzero
    \end{equation}
    where $\vv$ is an element of the subdifferential of $N(\cdot)$ at $\hbeta(\vy)$. 
    Recall that the subdifferential of any norm  $N(\cdot)$ at $\hbeta(\vy)\ne 0$ is the set
    $\{\vv\in\R^p: N(\vv) = 1 \text{ and } \vv^T\hbeta(\vy) = N(\hbeta(\vy)) \}$ \cite[Section 2.6]{ATF}.
    Therefore, taking an inner product of \eqref{moreau} with $\hbeta(\vy)$ yields
    \begin{align*}
    (\design\hbeta(\vy))^T (\vy - \design\hbeta(\vy))
    =
    nN(\hbeta(\vy))
    &=
    n \max_{\vw\in\R^p:N_\circ(\vw) = 1} \hbeta(\vy)^T\vw
    \\
    &\ge 
    \max_{\vu\in\R^n:N_\circ(\design^T\vu) = 1/n} (\design\hbeta(\vy))^T\vu= \max_{\vu\in C} (\design\hbeta(\vy))^T\vu.
    \end{align*}
    This proves \eqref{obtu} with $\vtheta = \vy - \design\hbeta(\vy)$. Thus, 
    we have established that $\design\hbeta(\vy) = \vy - P_C(\vy)$.

To prove part  (ii) of the proposition, we use that, for any closed
    convex subset $C$ of $\R^n$ and any $\vy,\vy'\in \R^n$,
    $$
     |P_C(\vy)- P_C(\vy')|_2^2 \le \langle P_C(\vy)- P_C(\vy'), \vy- \vy'\rangle,
    $$
    see, e.g., \cite[Proposition 3.1.3]{hiriart_lemarechal}.  This immediately implies 
 $$
     |\vy - P_C(\vy)- (\vy' -P_C(\vy'))|_2^2 \le  |\vy- \vy' |_2^2- |P_C(\vy)- P_C(\vy')|_2^2.
  $$   
  Part (ii) of the proposition follows now from part (i) and the last display.
     \qed
\end{proof}
    
We note that Propostion \ref{prop2} generalizes to any norm $N(\cdot)$ an analogous result obtained for the $\ell_1$-norm in \cite[Lemma 3]{tibshirani2012}.  

 We now turn to the general case assuming that $F$ is any convex penalty.

\begin{proposition}\label{prop3}
    Let $F: \R^p\rightarrow [0,+\infty]$ be any proper convex function.
    For all $\vy\in\R^n$, let $\hbeta(\vy)$ be any solution of the convex minimization problem \eqref{eq:def-hbeta-pen}.
    Then the estimator $\hmu(\vy) = \design\hbeta(\vy)$
    satisfies \eqref{eq:lipschitz-hmu}
    with $L=1$.
\end{proposition}
\begin{proof}
    \smartqed
    Let $\vy,\vy'\in\R^n$.
    The case $\design\hbeta(\vy) = \design\hbeta(\vy')$ is trivial so we assume in the following that
    $\design\hbeta(\vy) \ne \design\hbeta(\vy')$.
    The optimality condition and the Moreau-Rockafellar Theorem \cite[Theorem 3.30]{peypouquet2015convex}
    yield that
    there exist an element $\vh\in\R^p$ of the subdifferential $\partial F(\hbeta(\vy))$ of $F(\cdot)$ at $\hbeta(\vy)$
    and $\vh'\in \partial F(\hbeta(\vy'))$ such that
    \begin{equation*}
       \frac1{n}   \design^T(\design\hbeta(\vy) - \vy) + \vh= \vzero, \quad \text{and} \quad 
        \frac1{n}  \design^T(\design\hbeta(\vy') - \vy') + \vh' = \vzero.
    \end{equation*}
    Taking the difference of these two equalities we obtain
    \begin{equation*}
        \design^T(\design\hbeta(\vy) - \design \hbeta(\vy')) - \design^T(\vy - \vy')
        = n (\vh' - \vh).
    \end{equation*}
    Let $\vdelta = \hbeta(\vy) - \hbeta(\vy')$. Since $F$ is a proper convex function, we have that 
    $\vdelta^T(\vh - \vh')=\langle \vh - \vh', \hbeta(\vy) - \hbeta(\vy')\rangle \ge 0$ for any $\vh\in \partial F(\hbeta(\vy))$ and any $\vh'\in \partial F(\hbeta(\vy'))$
    (see, e.g., \cite[Proposition 3.22]{peypouquet2015convex}). Therefore,
    \begin{equation*}
        \vdelta^T \design^T\design \vdelta - \vdelta^T \design^T(\vy - \vy')
        =
       n \vdelta^T(\vh' - \vh) \le 0.
    \end{equation*}
    This and the Cauchy-Schwarz inequality yield
    \begin{equation}
        \norma{\design\vdelta}^2_2
        \le 
        \vdelta^T \design^T(\vy - \vy')
        \le \norma{\design\vdelta}_2 \norma{\vy - \vy'}_2,
    \end{equation}
    which implies $\norma{\design\vdelta}_2\le \norma{\vy - \vy'}_2$ since $\design\hbeta(\vy) \ne \design\hbeta(\vy')$.
    \qed
\end{proof}
Combining the above two propositions we obtain the following theorem.

\begin{theorem}\label{cor3}
    Let $R\ge0$ be a constant and $\vf\in\R^n$.
    Assume that $\vxi\sim \mathcal N(\vzero,\sigma^2 I_{n\times n})$ and let $\vy = \vf + \vxi$.
    Let $F: \R^p\rightarrow [0,+\infty]$ be any proper convex function.
    Assume also that the estimator $\hbeta$ defined in \eqref{eq:def-hbeta-norm-N}
    satisfies 
    \begin{equation}\label{rough_bound}
        \Pro\left(\norm{\design\hbeta(\vy) - \vf} \le R\right) \ge 1/2,
     \end{equation}
     or equivalently, the median of the prediction error
     satisfies $\Med[\norm{\design\hbeta(\vy) - \vf}]\le R$.
         Then,
    for all $t>0$,     
    \begin{equation}\label{cor3_3}
        \Pro\left(\norm{\design\hbeta(\vy) - \vf} \le 
            R
            +  \frac{\sigma t}{ \sqrt{n}}
\right) \ge \Phi(t)
    \end{equation}
and consequently,
     for all $x>0$,     
    \begin{equation}\label{cor3_3a}
        \Pro\left(\norm{\design\hbeta(\vy) - \vf} \le 
            R
            + \sigma \sqrt{2 x / n}\right) \ge 1-e^{-x}.
    \end{equation}
     Furthermore,
     \begin{equation}\label{integration}
        \E \norm{\design\hbeta(\vy) - \vf}
        \le
        R
        + \frac{\sigma}{ \sqrt{2 \pi n}}.
    \end{equation}
\end{theorem}
\begin{proof}
    \smartqed
    Fix $\vf\in\R^n$ and let $\vz\sim\mathcal N(\vzero,I_{n\times n})$. 
    Proposition~\ref{prop3} implies that  the function 
    $f(\vz) = \norm{\design\hbeta(\vf+\sigma \vz) - \vf}$
    satisfies
    \eqref{eq:lipschitz-concentration} 
    with $L=1$ for all $x>0$. 
    Thus, we can apply Proposition~\ref{prop1} and \eqref{cor3_3a}  follows from \eqref{eq:lipschitz-concentration}. 
   The bound \eqref{cor3_3a} is a simplified version of \eqref{cor3_3} using that $1-\Phi(t)\le e^{-t^2/2}, \ \forall t>0$.
    Finally, inequality \eqref{integration} is obtained by integration of \eqref{cor3_3}.
    \qed
\end{proof}

Note that condition \eqref{rough_bound} in Theorem \ref{cor3} is a weak property. To satisfy it, is enough to have a rough bound on $\norm{\design\hbeta(\vy) - \vf}$. Of course, we would like to have a meaningful value of $R$. In the next two sections, we give examples of such $R$. Namely, we show that Theorem \ref{cor3} allows one to derive sharp oracle inequalities for the prediction risk of such estimators as LASSO, group LASSO, and SLOPE. 

\begin{remark}Along with the concentration around the median, the prediction error $\norm{\design\hbeta(\vy) - \vf}$ also concentrates around its expectation. Using the Lipschitz property of Proposition~\ref{prop1}, and the theorem about Gaussian concentration with respect to the expectation (cf., e.g., \cite[Theorem B.6]{giraud2014introduction}) we find that, if $F:\R^n\rightarrow [0,+\infty]$ is a proper convex function, 
 \begin{equation}\label{cor3_3b}
        \Pro\left(\norm{\design\hbeta(\vy) - \vf} \le 
            \E \norm{\design\hbeta(\vy) - \vf}
            + \sigma \sqrt{2 x / n}\right) \ge 1-e^{-x}
    \end{equation}
    and 
    \begin{equation}\label{cor3_3bb}
        \Pro\left(\norm{\design\hbeta(\vy) - \vf} \ge 
            \E \norm{\design\hbeta(\vy) - \vf}
            - \sigma \sqrt{2 x / n}\right) \ge 1-e^{-x}.
    \end{equation}
For the special case of identity design matrix $\design=I_{n\times n}$, these properties have been proved in
\cite{vandegeer2015concentration} where they were applied to some problems of nonparametric estimation.
However, the bounds \eqref{cor3_3b} and \eqref{cor3_3bb} dealing with the concentration around the expectation
are of no use for the purposes of the present paper since initially we have no control of the expectation. On the other hand, 
a meaningful control of the median is often easy to obtain as shown in the examples below. This is the reason why we focus on the concentration around the median.
\end{remark}

\section{Application to LASSO}
\label{s:lasso}

The LASSO is a convex regularized estimator defined by the relation
\begin{equation}
    \hbeta \in \argmin_{\vbeta\in\R^p} \left(\norm{\vy - \design\vbeta}^2 + 2\lambda \norma{\vbeta}_1 \right)
    \label{eq:def-lasso}
\end{equation}
where $\lambda>0$ is a tuning parameter.
Risk bounds for the LASSO estimator are established under some conditions on the design matrix $\design$
that measure the correlations between its columns.
The Restricted Eigenvalue (RE) condition \cite{bickel2009simultaneous} is defined as follows.
For any $S\subset\{1,...,p\}$ and $c_0> 0$,
we define the Restricted Eigenvalue constant $\kappa(S,c_0)\ge 0$ by the formula
\begin{equation}
    \kappa^2(S,c_0) \triangleq \min_{\vdelta\in\R^p: |\vdelta_{S^c}|_1\le c_0|\vdelta_{S}|_1} \frac{\norm{\design\vdelta}^2}{|\vdelta|_2^2}.
    \label{eq:def-kappa}
\end{equation}
She RE condition $RE(S,c_0)$ is said to hold if $\kappa(S,c_0) > 0$. Note that \eqref{eq:def-kappa} is slightly different from the original definition given in \cite{bickel2009simultaneous} since we have $\vdelta$ and not $\vdelta_S$ in the denominator. However, the two definitions are equivalent up to a constant depending only on $c_0$, cf. \cite{bellec2016slope}.
In terms of the Restricted Eigenvalue constant, we have the following deterministic result.
\begin{proposition}
    \label{prop:lasso-deterministic}
    Let $\lambda>0$ be a tuning parameter.
    On the event 
    \begin{equation}
        \left\{ \frac 1 n \norma{\design^T \vxi}_\infty \le \frac \lambda 2 \right\},
        \label{eq:lasso-event}
    \end{equation}
    the LASSO estimator \eqref{eq:def-lasso} with tuning parameter $\lambda$ satisfies
    \begin{equation}
        \norm{
            \design\hbeta
            -
            \vf
        }^2
        \le
        \min_{S\subset \{1,...,p\} }
               \left[
        \min_{\vbeta\in\R^p: \supp(\vbeta)=S}
             \norm{
                \design\vbeta
                -
                \vf
            }^2
            + \frac{9 |S| \lambda^2}{4\kappa^2(S,3)}
        \right]
        \label{eq:soi-squared-lasso}
    \end{equation}
    with the convention that $a/0=+\infty$ for any $a>0$.
\end{proposition}
An oracle inequality as in Proposition \ref{prop:lasso-deterministic} has been first established in \cite{bickel2009simultaneous}
with a multiplicative constant greater than $1$ in front of the right hand side of \eqref{eq:soi-squared-lasso}.
The fact that this constant can be reduced to 1, so that the oracle inequality becomes sharp, is due to \cite{koltchinskii2011nuclear}. 
For the sake of completeness, we provide below a sketch of the proof of
Proposition \ref{prop:lasso-deterministic}.
\begin{proof}
    \smartqed
    We will use the following fact \cite[Lemma A.2]{bellec2016slope}.
    \begin{lemma}\label{lem1}
    Let $F:\R^p\rightarrow\R$ be a convex function,
    let $\vf, \vxi\in \R^n$, $\vy=\vf+\vxi$ and let $\design$ be any $n\times p$ matrix.
    If $\hbeta$ is a solution of the minimization problem \eqref{eq:def-hbeta-pen},
        then $\hbeta$ satisfies, for all $\vbeta\in\R^p$,
    \begin{equation*}
        \norm{\design\hbeta - \vf}^2
        -
        \norm{\design\vbeta - \vf}^2
        \le
        2\left(\frac 1 n \vxi^T\design(\hbeta - \vbeta)
        + F(\vbeta)
        - F(\hbeta)\right)
        - \norm{\design(\hbeta - \vbeta)}^2.
        \label{eq:almost-sure}
    \end{equation*}
\end{lemma}

    Let $S\subset\{1,...,p\}$ and $\vbeta$ be minimizers of the right hand side of
    \eqref{eq:soi-squared-lasso} and let $\vdelta = \hbeta - \vbeta$. We will assume that $\kappa(S,3)>0$ since otherwise the claim is trivial.
    From Lemma \ref{lem1} with $F(\vbeta)=\lambda  |\vbeta|_1$ we have
    \begin{equation*}
        \norm{
            \design\hbeta
            -
            \vf
        }^2
        -
        \norm{
            \design\vbeta
            -
            \vf
        }^2
        \le 2\left( \tfrac 1 n \vxi^T\design \vdelta + \lambda |\vbeta|_1 - \lambda |\hbeta|_1\right) 
        - \norm{\design\vdelta}^2
        \triangleq D.
    \end{equation*}
    On the event \eqref{eq:lasso-event}, using the duality inequality $\vx^T\vdelta \le |\vx|_\infty |\vdelta|_1$
    valid for all $\vx,\vdelta\in\R^p$
    and the triangle inequality for the $\ell_1$-norm, we find that
    the right hand side of the previous display satisfies
    \begin{equation*}
        D
        \le
        2\lambda \left[
            \frac 1 2|\vdelta|_1   
            + |\vbeta|_1 - |\hbeta|_1
        \right] - \norm{\design\vdelta}^2
        \le 
        2\lambda \left[
            \frac 3 2 |\vdelta_S|_1   
            - \frac 1 2 |\vdelta_{S^c}|_1
        \right] - \norm{\design\vdelta}^2.
    \end{equation*}
    If $|\vdelta_{S^c}|_1 > 3|\vdelta_S|_1$ then the claim follows trivially.
    Otherwise, if $|\vdelta_{S^c}|_1 \le 3|\vdelta_S|_1$ we have
    $|\vdelta|_2 \le \norm{\design\vdelta}/\kappa(S,3)$
    and thus,
    by the Cauchy-Schwarz inequality,
    \begin{equation}\label{eq:last}
        3\lambda|\vdelta_S|_1
        \le
        3 \lambda\sqrt{|S|} |\vdelta_S|_2
        \le 
        \frac{9|S|\lambda^2}{4\kappa^2(S,3)} + \norm{\design\vdelta}^2.
    \end{equation}
    Combining the above three displays yields \eqref{eq:soi-squared-lasso}.
    \qed
\end{proof}

\begin{theorem}
    \label{thm:lasso}
    Let $p\ge 2$ and $\lambda\ge 2 \sigma \sqrt{2\log(p)/n}$.
    Assume that the diagonal elements of matrix $\frac 1 n
    \design^T \design$ are not greater than 1.
    Then for any $\delta\in(0,1)$, the LASSO estimator \eqref{eq:def-lasso} with tuning parameter $\lambda$
    satisfies
    \begin{equation}
        \norm{
            \design\hbeta
            -
            \vf
        }
        \le
        \min_{S\subset \{1,...,p\}}
         \left[
        \min_{\vbeta\in\R^p: \supp(\vbeta)=S}
        \norm{
            \design\vbeta
            -
            \vf
        }
        + \frac{3\lambda\sqrt{|S|}}{2\kappa(S,3)}  \right]
        + \frac{\sigma \Phi^{-1}(1-\delta)}{\sqrt{n}} 
        \label{eq:all-delta-lasso}
    \end{equation}
    with probability at least $1-\delta$, noting that $\Phi^{-1}(1-\delta)\le \sqrt{2\log (1/\delta)}$. Furthermore,
    \begin{equation}
      \E  \norm{
            \design\hbeta
            -
            \vf
        }
        \le
        \min_{S\subset \{1,...,p\}}        
        \left[
        \min_{\vbeta\in\R^p: \supp(\vbeta)=S}
        \norm{
            \design\vbeta
            -
            \vf
        }
        + \frac{3\lambda\sqrt{|S|}}{2\kappa(S,3)}
                \right]
        + \frac{\sigma}{ \sqrt{2 \pi n}}.
        \label{eq:all-delta-lasso-bis}
         \end{equation}
\end{theorem}
Theorem \ref{thm:lasso} is a simple consequence of Proposition \ref{prop:lasso-deterministic} and Theorem \ref{cor3}. 
Its proof is given at the end of the present section.

Previous works on the LASSO estimator established that
for some fixed $\delta_0\in (0,1)$ 
the estimator  \eqref{eq:def-lasso} with tuning parameter $\lambda= c_1 \sigma \sqrt{2\log(c_2p/\delta_0)}$, where  $c_1>1, c_2\ge 1$ are some constants, 
satisfies an oracle inequality of the form \eqref{eq:soi-squared-lasso}
with probability at least $1-\delta_0$,
see for instance \cite{bickel2009simultaneous,koltchinskii2011nuclear,dalalyan2014prediction}
or the books on high-dimensional statistics \cite{giraud2014introduction,buhlmann2011statistics,vdgeer}.
Thus, such oracle inequalities 
were available only for one fixed confidence level $1-\delta_0$
tied to the tuning parameter $\lambda$. 
Remarkably, Theorem \ref{thm:lasso} shows that the LASSO estimator with a universal (not level-dependent) tuning parameter, which can be as small as $2 \sigma \sqrt{(2\log p)/n}$, satisfies \eqref{eq:all-delta-lasso} for all confidence levels $\delta\in(0,1)$. As a consequence, we can obtain an oracle inequality \eqref{eq:all-delta-lasso-bis} for the expected error, while control of the expected error was not achievable with the previously known methods of proof. Furthermore, bounds for any moments of the prediction error can be readily obtained by integration of \eqref{eq:all-delta-lasso}. Analogous facts have been shown recently in \cite{bellec2016slope}
using different techniques.
To our knowledge, the present paper and \cite{bellec2016slope}
provide the first evidence of such properties of the LASSO estimator. 

In addition, Theorem \ref{thm:lasso} shows that the rate of concentration in the oracle inequalities   
is better than it was commonly understood before. Let $S\subset\{1,...,p\}$, $s=|S|$
and set $\delta = \exp(-2 s \lambda^2n/ \sigma^2\kappa^2(S,3) )$.
For this choice of $\delta$, Theorem \ref{thm:lasso} implies that 
if $\lambda\ge 2\sigma\sqrt{2\log(p)/n}$ then
\begin{equation*}
    \mathbb P
    \left(
        \norm{
            \design\hbeta
            -
            \vf
        }
        \le
        \min_{\vbeta\in\R^p: \supp(\vbeta)=S}
        \norm{
            \design\vbeta
            -
            \vf
        }
        + \frac{7\lambda\sqrt{|S|}}{2\kappa(S,3)}
    \right) \ge 
    1 - e^{-2 s \lambda^2 n/ \sigma^2\kappa^2(S,3) }
    .
\end{equation*}
Since the diagonal elements of $\frac 1 n \design^T\design$ are at most
1, we have $\kappa(S,3) \le 1$.
Thus, the probability on the right hand side of the last display is at least $1-p^{-16s}$.
Interestingly, this probability depends on the sparsity $s$ and tends to 1 exponentially fast as $s$ grows.
The previous proof techniques
\cite{bickel2009simultaneous,koltchinskii2011nuclear,dalalyan2014prediction,buhlmann2011statistics} 
provided, for the same type of probability, only an estimate of the form $1-p^{-b}$
for some fixed constant $b>0$ independent of $s$.

The oracle inequality \eqref{eq:soi-squared-lasso} holds for the error $\norm{\design\hbeta - \vf}$.
In order to obtain an oracle inequality for the squared error $\norm{\design\hbeta - \vf}^2$,
one can combine \eqref{eq:all-delta-lasso}
with the basic inequality $(a+b+c)^2\le 3(a^2+b^2+c^2)$.
This yields that under the assumptions of Theorem \ref{thm:lasso}, 
the LASSO estimator with tuning parameter $\lambda\ge 2\sigma\sqrt{2\log(p)/n}$ satisfies,
with probability at least $1-\delta$,
\begin{equation*}
    \norm{
        \design\hbeta
        -
        \vf
    }^2
    \le
    \min_{S\subset \{1,...,p\}}  \left[
    \min_{\vbeta\in\R^p: \supp(\vbeta)=S}
          3 \norm{
            \design\vbeta
            -
            \vf
        }^2
        + \frac{27\lambda^2|S|}{4\kappa^2(S,3)} \right]
        + \frac{6\log(1/\delta)}{n}
   .
\end{equation*}
The constant $3$ in front of the
$
\norm{
    \design\vbeta
    -
    \vf
}^2$
can be reduced to 1 using the techniques developed in \cite{bellec2016slope}.

{\it Proof of Theorem \ref{thm:lasso}.}
    The random variable $|\design^T\vxi|_\infty / \sqrt n$
    is the maximum of $p$ centered Gaussian random variables
    with variance at most $\sigma^2$. 
    If $\eta\sim\mathcal N(0,1)$,
    a standard approximation of the Gaussian tail gives
        $\mathbb P(|\eta|> x) \le \sqrt{2/\pi} (e^{-x^2/2}/x)$ for all $x>0$.
    This approximation with $x=\sqrt{2\log p}$ together with
    and the union bound imply
    that the event \eqref{eq:lasso-event} with $\lambda\ge 2\sigma \sqrt{(2\log p)/n}$ has probability
    at least
    $1 - 1/\sqrt{\pi\log p}$,
    which is greater than $1/2$ for all $p\ge 3$. For $p=2$, the probability of this event is bounded from below by $1-2\mathbb P(|\eta|> \sqrt{2\log 2})>1/2$. Thus, Proposition \ref{prop:lasso-deterministic} implies that condition \eqref{rough_bound} is satisfied with $R$ being the square root of the right hand side of \eqref{eq:soi-squared-lasso}.
    Let $S$ and $\vbeta$ be minimizers of the right hand side of 
    \eqref{eq:all-delta-lasso}.
    Applying Theorem \ref{cor3} and the inequality $\sqrt{a+b}\le \sqrt a + \sqrt b$ with
    $\sqrt a = \norm{\design\vbeta - \vf}$
    and
    $\sqrt b = 3\lambda\sqrt{|S|}/(2\kappa(S,3))$
    completes the proof.

\section{Application to group LASSO and related penalties}
\label{s:glasso}

The above arguments can be used to establish oracle inequalities for the group LASSO estimator similar to those obtained in Section \ref{s:lasso} for the usual LASSO. The improvements as compared to the previously known oracle bounds (see, e.g., \cite{lptv2011} or the books \cite{giraud2014introduction,buhlmann2011statistics,vdgeer}) are the same as above -- independence of the tuning parameter on the confidence level, better concentration, and derivation of bounds in expectation.  

Let $G_1,...,G_M$ be a partition of $\{1,...,p\}$. The group LASSO estimator is a solution of the convex minimization problem
\begin{equation}
    \hbeta \in \argmin_{\vbeta\in\R^p}\left(
    \norm{\vy-\design\vbeta}^2
    + 2  \lambda \sum_{k=1}^M |\vbeta_{G_k}|_2\right),
    \label{eq:def-glasso}
\end{equation}
 where $\lambda>0$ is a tuning parameter. In the following, we assume that the groups $G_k$ have the same cardinality $|G_k|= T=p/M$, $k=1,...,M$. 
 
 We will need the following group analog of the RE constant introduced in~\cite{lptv2011}. For any $S\subset\{1,...,M\}$ and $c_0> 0$,
we define the group Restricted Eigenvalue constant $\kappa_G(S,c_0)\ge 0$ by the formula
\begin{equation}
    \kappa^2_G(S,c_0) \triangleq \min_{\vdelta\in\mathcal C(S,c_0)} \frac{\norm{\design\vdelta}^2}{|\vdelta|_2^2},
    \label{eq:def-kappa-g}
\end{equation}
where $\mathcal C(S,c_0)$ is the cone
$$
\mathcal C (S,c_0)\triangleq \{\vdelta\in\R^p: \sum_{k\in S^c}|\vdelta_{G_k}|_2\le c_0\sum_{k\in S}|\vdelta_{G_k}|_2\}.
$$
Denote by $\design_{G_k}$ the $n\times |G_k|$ submatrix of $\design$ composed from all the columns of $\design$ with indices in $G_k$. For any $\vbeta\in \R^p$, set ${\mathcal K}(\vbeta)=\{k\in\{1,...,M\}: \vbeta_{G_k}\ne \vzero\}$. The following deterministic result holds. 
\begin{proposition}
    \label{prop:glasso-deterministic}
    Let $\lambda>0$ be a tuning parameter.
    On the event 
    \begin{equation}
        \left\{ \max_{k=1,...,M} \,\frac 1 n |\design_{G_k}^T \vxi |_2 \le \frac \lambda 2 \right\},
        \label{eq:glasso-event}
    \end{equation}
    the group LASSO estimator \eqref{eq:def-glasso} with tuning parameter $\lambda$ satisfies
    \begin{equation}
        \norm{
            \design\hbeta
            -
            \vf
        }^2
        \le
        \min_{S\subset \{1,...,M\} }
               \left[
        \min_{\vbeta\in\R^p: {\mathcal K}(\vbeta)=S}
             \norm{
                \design\vbeta
                -
                \vf
            }^2
            + \frac{9 |S| \lambda^2}{4\kappa^2_G(S,3)}
        \right]
        \label{eq:soi-squared-glasso}
    \end{equation}
    with the convention that $a/0=+\infty$ for any $a>0$.
\end{proposition}
\begin{proof} We follow the same lines as in the proof of Proposition~\ref{prop:lasso-deterministic}. The difference is that 
we replace the $\ell_1$ norm by the group LASSO norm $\sum_{k=1}^M |\vbeta_{G_k}|_2$,   and the value $D$ now has the form
 \begin{eqnarray*} 
 D&=&2\left( \tfrac 1 n \vxi^T\design \vdelta + \lambda\sum_{k=1}^M |\vbeta_{G_k}|_2 - \lambda \sum_{k=1}^M |\hbeta_{G_k}|_2\right) 
        - \norm{\design\vdelta}^2.
         \end{eqnarray*}
            Then, on the event \eqref{eq:glasso-event} we obtain
             \begin{eqnarray*}
        D&\le& 2\lambda\left( \frac 1 2 \sum_{k=1}^M |\vdelta_{G_k}|_2 + \sum_{k=1}^M |\vbeta_{G_k}|_2 -\sum_{k=1}^M |\hbeta_{G_k}|_2\right) 
        - \norm{\design\vdelta}^2\\
        &\le& 2\lambda\left( \frac 3 2 \sum_{k\in S} |\vdelta_{G_k}|_2  -\frac 1 2 \sum_{k\in S^c}|\vdelta_{G_k}|_2 \right) 
        - \norm{\design\vdelta}^2,
           \end{eqnarray*}
           where the last inequality uses the fact that ${\mathcal K}(\vbeta)=S$. The rest of the proof is quite analogous to that of Proposition~\ref{prop:lasso-deterministic} if we replace there
$\kappa(S,3)$ by $\kappa_G(S,3)$.
\qed
            \end{proof}
            
            To derive the oracle inequalities for group LASSO, we use the same argument as in the case of LASSO. In order to apply Theorem~\ref{cor3},
we need to find a ``weak bound" $R$ on the error $\norm{\design\hbeta - \vf}$, i.e., a bound valid
with probability at least $1/2$. The next lemma gives a range of values of $\lambda$ such that the event \eqref{eq:glasso-event} holds with probability at least 1/2.  Then, due to Proposition~\ref{prop:glasso-deterministic}, we can take as $R$ the square root of the right hand side of \eqref{eq:soi-squared-glasso}. 

Denote by $\|\design_{G_k}\|_{\rm sp}\triangleq \sup_{|\vx|_2\le 1}|\design_{G_k}\vx|_2$ the spectral norm of matrix $\design_{G_k}$, and set $\psi^*  =\max_{k=1,...,M}\|\design_{G_k}\|_{\rm sp}/\sqrt{n}$.

\begin{lemma} \label{lem:group}
Let the diagonal elements of matrix $\frac 1 n
    \design^T \design$ be not greater than 1. If 
\begin{equation}\label{lambda:glasso}
\lambda \ge \frac{2\sigma}{\sqrt{n}}\left(\sqrt{T} + \psi^* \sqrt{2\log(2 M)}\right),
\end{equation}  
then the event \eqref{eq:glasso-event} has probability at least 1/2.
\end{lemma}
\begin{proof} Note that the function $\vu\mapsto |\design_{G_k}\vu|_2$ is $\psi^*\sqrt{n}$-Lipschitz with respect to the Euclidean norm. Therefore, the Gaussian concentration inequality, cf., e.g., \cite[Theorem B.6]{giraud2014introduction}, implies that, for all $x>0$,  
$$
\Pro \left( |\design_{G_k}\vxi|_2 \ge \E |\design_{G_k}\vxi|_2 + \sigma\psi^*\sqrt{2xn} \right)\le e^{-x}, \qquad k=1,...,M. 
$$
Here, $ \E |\design_{G_k}\vxi|_2\le \left( \E |\design_{G_k}\vxi|_2^2\right)^{1/2}=\sigma \|\design_{G_k}\|_F$, where $\|\cdot\|_F$ is the Frobenius norm. By the assumption of the lemma, all columns of $\design$ have the Euclidean norm at most~$\sqrt{n}$. Since $\design_{G_k}$ is composed from $T$ columns of  $\design$ we have $\|\design_{G_k}\|_F^2 \le nT$, so that $ \E |\design_{G_k}\vxi|_2\le\sigma \sqrt{nT}$ for $k=1,...,M$. Thus, for all $x>0$, 
$$
\Pro \left( |\design_{G_k}\vxi|_2 \ge \sigma (\sqrt{nT} +\psi^*\sqrt{2xn}) \right)\le e^{-x}, \qquad k=1,...,M,
$$
and the result of the lemma follows by application of the union bound.
\qed
\end{proof}

Combining Proposition~\ref{prop:glasso-deterministic}, Lemma~\ref{lem:group} and Theorem~\ref{cor3} we get the following result. 

\begin{theorem}
    \label{thm:glasso}
    Assume that the diagonal elements of matrix $\frac 1 n
    \design^T \design$ are not greater than 1. Let $\lambda$ be such that \eqref{lambda:glasso} holds.
    Then for any $\delta\in(0,1)$, the group LASSO estimator \eqref{eq:def-glasso} with tuning parameter $\lambda$
    satisfies, with probability at least $1-\delta$, the oracle inequality \eqref{eq:all-delta-lasso} with $\kappa(S,3)$ replaced by $\kappa_G(S,3)$, and $\supp(\vbeta)$ replaced by ${\mathcal K}(\vbeta)$. Furthermore, it satisfies \eqref{eq:all-delta-lasso-bis} with the same modifications.
\end{theorem}

We can consider in a similar way a more general class of penalties generated by cones \cite{micchelli}. Let ${\mathcal A}$ be a convex cone in $(0,+\infty)^p$. For any $\beta\in \R^p$, set
\begin{equation}
\|\beta\|_{\mathcal A} \triangleq \inf_{a\in {\mathcal A}} \frac12 \sum_{j=1}^p\left(\frac{\beta_j^2}{a_j} +a_j\right)= \inf_{a\in {\mathcal A}: |a|_1\le 1}\sqrt{\sum_{j=1}^p\frac{\beta_j^2}{a_j}} 
 \label{eq:cone-penalty}
         \end{equation}
and consider the penalty $F(\beta)=\lambda \|\beta\|_{\mathcal A}$ where $\lambda>0$. The function $\|\cdot\|_{\mathcal A}$ is convex since it is a minimum of a convex function of the couple $(\beta,a)$ over $a$ in a convex set \cite[Corollary 2.4.5]{hiriart_lemarechal}. In view of its positive homogeneity, it is also a norm.
The group LASSO penalty is a special case of \eqref{eq:cone-penalty} corresponding to the cone of all vectors $a$ with positive components that are constant on the blocks $G_k$ of the partition. Many other interesting examples are given in \cite{micchelli,maurer_pontil}, see also \cite[Section 6.9]{vdgeer}. 

Such penalties induce a class of admissible sets of indices $S\subset \{1,...,p\}$. This is a generalization of the sets of indices corresponding to vectors $\beta$ that vanish on entire blocks in the case of group LASSO.  Roughly speaking, the set of indices $S$ would be suitable for our construction if, for any $a\in {\mathcal A}$, the vectors $a_S$ and $a_{S^c}$ belong to ${\mathcal A}$. 
However, this is not possible since, by definition, the elements of ${\mathcal A}$ must have positive components. Thus, we slightly modify this condition on~$S$. A~set $S\subset \{1,...,p\}$ will be called {\it admissible} with respect to ${\mathcal A}$ if, for any $a\in {\mathcal A}$ and any $\epsilon>0$, there exist vectors $b_{S^c}\in \R^p$ and $b_{S}\in \R^p$ supported on $S^c$ and $S$ respectively with all components in $(0,\epsilon)$ and such that 
$a_S+b_{S^c}\in {\mathcal A}$, and $a_{S^c}+b_{S}\in {\mathcal A}$. 

The following lemma shows that, for admissible $S$, the norm $\|\cdot\|_{\mathcal A}$ has the same decomposition property as the $\ell_1$ norm. 

\begin{lemma}\label{lem2}
If $S\subset \{1,...,p\}$ is an {\it admissible} set of indices with respect to ${\mathcal A}$, then 
$$
\|\beta\|_{\mathcal A} = \|\beta_S\|_{\mathcal A} + \|\beta_{S^c}\|_{\mathcal A}.
$$
\end{lemma}
\begin{proof}
As $\|\cdot\|_{\mathcal A} $ is a norm, we have to show only that $
\|\beta\|_{\mathcal A} \ge \|\beta_S\|_{\mathcal A} + \|\beta_{S^c}\|_{\mathcal A}.
$ Obviously, 
\begin{equation}
\|\beta\|_{\mathcal A} \ge \inf_{a\in {\mathcal A}} \frac12 \sum_{j\in S}\left(\frac{\beta_j^2}{a_j} +a_j\right)+\inf_{a\in {\mathcal A}} \frac12 \sum_{j\in S^c}\left(\frac{\beta_j^2}{a_j} +a_j\right).
 \label{eq:lem1}
         \end{equation}
Since $S$ is admissible, adding the sum $\sum_{j\in S^c}a_j$ under the infimum in the first term on the right hand side does not change the result:
\begin{eqnarray*}
\inf_{a\in {\mathcal A}} \frac12 \sum_{j\in S}\left(\frac{\beta_j^2}{a_j} +a_j\right) = \inf_{a\in {\mathcal A}} \frac12 \left[\sum_{j\in S}\left(\frac{\beta_j^2}{a_j} +a_j\right) + \sum_{j\in S^c}a_j \right] =  \|\beta_S\|_{\mathcal A}.
\end{eqnarray*}
The second term on the right hand side of \eqref{eq:lem1} is treated analogously.
\qed
\end{proof}
Next, for any $S\subset\{1,...,p\}$ and $c_0> 0$,
we need an analog of the RE constant corresponding to the penalty $\|\cdot\|_{\mathcal A}$, cf. \cite{vdgeer}. We define
$q_{\mathcal A}(S,c_0)\ge 0$ by the formula
\begin{equation}
    q^2_{\mathcal A}(S,c_0) \triangleq \min_{\vdelta\in\mathcal C'(S,c_0)} \frac{\norm{\design\vdelta}^2}{\|\vdelta_S\|_{\mathcal A}^2},
    \label{eq:def-kappa-cone}
\end{equation}
where $\mathcal C'(S,c_0)$ is the cone
$$
\mathcal C' (S,c_0)\triangleq \{\vdelta\in\R^p: \|\vdelta_{S^c}\|_{\mathcal A}\le c_0\|\vdelta_{S}\|_{\mathcal A}\}.
$$
As in the previous examples, our starting point will be  a deterministic bound that holds on a suitable event. This result  is analogous to Propositions \ref{prop:lasso-deterministic} and \ref{prop:glasso-deterministic}. To state it, we define 
$$
\|\vbeta\|_{\mathcal A, \circ} = \sup_{a\in {\mathcal A}: |a|_1\le 1}\sqrt{\sum_{j=1}^p a_j \beta_j^2} 
$$
which is the dual norm to $\|\cdot\|_{\mathcal A}$. 
\begin{proposition}
    \label{prop:cone-deterministic}
    Let ${\mathcal A}$ be a convex cone in $(0,+\infty)^p$, and let ${\mathbb S}_{\mathcal A}$ be set of all $S\subset \{1,...,p\}$ that are admissible with respect to ${\mathcal A}$.
    Let $\lambda>0$ be a tuning parameter.
    On the event 
    \begin{equation}
        \left\{\|\tfrac 1 n \design^T \vxi\|_{\mathcal A, \circ}  \le \frac \lambda 2 \right\},
        \label{eq:cone-event}
    \end{equation}
    the estimator \eqref{eq:def-hbeta-pen} with penalty $F(\cdot)=\lambda\|\cdot\|_{\mathcal A}$ satisfies
    \begin{equation}
        \norm{
            \design\hbeta
            -
            \vf
        }^2
        \le
        \min_{S\in {\mathbb S}_{\mathcal A}}
               \left[
        \min_{\vbeta\in\R^p: \supp (\vbeta)=S}
             \norm{
                \design\vbeta
                -
                \vf
            }^2
            + \frac{9 \lambda^2}{4 q^2_{\mathcal A}(S,3)}
        \right]
        \label{eq:soi-squared-cone}
    \end{equation}
    with the convention that $a/0=+\infty$ for any $a>0$.
\end{proposition}
\begin{proof} \smartqed In view of Lemma \ref{lem2}, we can follow exactly the lines of the proof of  Proposition~\ref{prop:lasso-deterministic} by replacing there the $\ell_1$ norm by the norm $\|\cdot\|_{\mathcal A}$ and taking into account the duality bound $\tfrac 1 n \vxi^T\design \Delta\le \|\tfrac 1 n \design^T \vxi\|_{\mathcal A, \circ} \|\Delta\|_{\mathcal A}$. At the end, instead of \eqref{eq:last}, we use that
 \begin{equation*}
        3\lambda\|\Delta_S\|_{\mathcal A}
        \le
        3 \lambda \frac{\norm{\design\vdelta}}{q_{\mathcal A}(S,3)} 
        \le 
        \frac{9\lambda^2}{4 q_{\mathcal A}^2(S,3)} + \norm{\design\vdelta}^2.
    \end{equation*}
\qed
\end{proof}
Our next step is to find a range of values of $\lambda$ such that the event \eqref{eq:cone-event} holds with probability at least 1/2. We will consider only the case when ${\mathcal A}$ is a polyhedral cone, which corresponds to many examples considered in \cite{micchelli,maurer_pontil}. We will denote by ${\mathcal A}'$ the closure of the set ${\mathcal A}\cap \{a: |a|_1\le 1\}$.
\begin{lemma} \label{lem:cone}
Let the diagonal elements of matrix $\frac 1 n
    \design^T \design$ be not greater than 1. Let ${\mathcal A}$ be a polyhedral cone, and let ${\mathcal E}_{\mathcal A'}$ be the set of extremal points of ${\mathcal A}'$. If  
\begin{equation}\label{lambda:cone}
\lambda \ge \frac{2\sigma}{\sqrt{n}}\left(1 +  \sqrt{2\log(2 |{\mathcal E}_{\mathcal A'}|)}\right),
\end{equation}  
then the event \eqref{eq:cone-event} has probability at least 1/2.
\end{lemma}
\begin{proof} Denote by $\eta_j=\tfrac 1 n e_j^T\design^T \vxi$ the $j$th component of $\tfrac 1 n \design^T \vxi$. We have 
\begin{equation}\label{proof:lambda:cone}
\|\tfrac 1 n \design^T \vxi\|_{\mathcal A, \circ}  = \sup_{a\in {\mathcal A}: |a|_1\le 1}\sqrt{\sum_{j=1}^p a_j \eta_j^2}= \max_{a\in {\mathcal E}_{\mathcal A'}} \sqrt{\sum_{j=1}^p a_j \eta_j^2},
\end{equation} 
where the last equality is due to the fact that ${\mathcal A'}$ is a convex polytope. Let $\vz=\vxi/\sigma$ be a standard normal $\mathcal N(\vzero,I_{n\times n})$ random vector. Note that, for all $a$ such that $|a|_1\le 1$, the function $f_a(\vz) = \sigma \sqrt{\sum_{j=1}^p a_j (\tfrac 1 n e_j^T\design^T \vz)^2}$ is $\sigma/\sqrt{n}$-Lipschitz with respect to the Euclidean norm. Indeed,
\begin{eqnarray*}
|f_a(\vz)-f_a(\vz')|&\le&{\sigma}\sqrt{\sum_{j=1}^p a_j (\tfrac 1 n e_j^T\design^T (\vz-\vz'))^2}\le  \frac{\sigma}{\sqrt{n}} \sqrt{\sum_{j=1}^p a_j \| \design e_j \|^2 |\vz-\vz'|_2^2}
\\
&\le & \frac{\sigma}{\sqrt{n}} |\vz-\vz'|_2,\qquad \forall \ |a|_1\le 1,
\end{eqnarray*}
since $\max_j  \| \design e_j \|^2\le 1$ by the assumption of the lemma. Therefore, the Gaussian concentration inequality, cf., e.g., \cite[Theorem B.6]{giraud2014introduction}, implies that, for all $x>0$,  
$$
\Pro \left( f_a(\vz)  \ge \E  f_a(\vz) + \sigma\sqrt{\frac{2x}{n}} \right)\le e^{-x}. 
$$
Here, $f_a(\vz)=\sqrt{\sum_{j=1}^p a_j \eta_j^2}$ and  $\E \sqrt{\sum_{j=1}^p a_j \eta_j^2}\le \left(\E \sum_{j=1}^p a_j \eta_j^2\right)^{1/2} \le \sigma/\sqrt{n}$ for all $a$ in the positive orthant such that $|a|_1\le 1$ where we have used that $\E\eta_j^2\le \sigma^2/n$ for $j=1,...,p$. Thus, for all $a$ in the positive orthant such that $|a|_1\le 1$ and all $x>0$ we have 
$$
\Pro \left( \sqrt{\sum_{j=1}^p a_j \eta_j^2} \ge  \frac{\sigma}{\sqrt{n}} ( 1+\sqrt{2x}) \right)\le e^{-x}. 
$$
The result of the lemma follows immediately from this inequality, \eqref{proof:lambda:cone} and the union bound.
\qed
\end{proof}
Finally, from  Proposition~\ref{prop:cone-deterministic}, Lemma~\ref{lem:cone} and Theorem~\ref{cor3} we get the following theorem. 

\begin{theorem}
    \label{thm:cone}
    Assume that the diagonal elements of matrix $\frac 1 n
    \design^T \design$ are not greater than 1. Let $\lambda$ be such that \eqref{lambda:cone} holds.
    Then for any $\delta\in(0,1)$, the estimator \eqref{eq:def-hbeta-pen} with penalty $F(\cdot)=\lambda\|\cdot\|_{\mathcal A}$ 
    satisfies
    \begin{equation*}
        \norm{
            \design\hbeta
            -
            \vf
        }
        \le
        \min_{S\in {\mathbb S}_{\mathcal A} }
               \left[
        \min_{\vbeta\in\R^p: \supp(\vbeta)=S}
             \norm{
                \design\vbeta
                -
                \vf
            }
            + \frac{3  \lambda}{2q_{\mathcal A}(S,3)}
        \right]
        + \frac{\sigma \Phi^{-1}(1-\delta)}{\sqrt{n}} 
        \label{eq:all-delta-cone}
    \end{equation*}
    with probability at least $1-\delta$. Furthermore,
    \begin{equation*}
      \E  \norm{
            \design\hbeta
            -
            \vf
        }
        \le
        \min_{S\in {\mathbb S}_{\mathcal A} }
               \left[
        \min_{\vbeta\in\R^p:\supp(\vbeta)=S}
             \norm{
                \design\vbeta
                -
                \vf
            }
            + \frac{3 \lambda}{2q_{\mathcal A}(S,3)}
        \right]
        + \frac{\sigma}{ \sqrt{2 \pi n}}.
        \label{eq:all-delta-cone-bis}
         \end{equation*}
\end{theorem}
Note that, in contrast to Theorems \ref{thm:lasso} and \ref{thm:glasso}, Theorem \ref{thm:cone} is a less explicit result. 
Indeed, the form of the oracle inequalities depends on the value $q_{\mathcal A}(S,3)$ and, through $\lambda$, on the value
$|{\mathcal E}_{\mathcal A'}|$. Both quantities are solutions of nontrivial geometric problems depending on the form of the cone ${\mathcal A}$.  Little is known about them. Note also that the knowledge of $|{\mathcal E}_{\mathcal A'}|$ (or of an upper bound on it) is required to find the appropriate~$\lambda$.

\section{Application to SLOPE}
\label{s:slope}

This section studies the SLOPE estimator introduced in \cite{bogdan}, which is yet another convex regularized estimator.
Define the norm $|\cdot|_*$ in $\R^p$ by the relation
\begin{equation*}
    |\vu|_* \triangleq \max_{\phi} \sum_{j=1}^p  \mu_j u_{\phi(j)} ,
    \qquad
    \vu=(u_1,...,u_p) \in\R^p,
\end{equation*}
where the maximum is taken over all permutations $\phi$ of $\{1,...,p\}$ and $\mu_j>0$ are some weights. In what follows,
we assume that
\begin{equation*}
    \mu_j = \sigma\sqrt{\log(2p/j) /n}, \qquad j=1,...,p.
\end{equation*}
For any $\vu=(u_1,...,u_p)\in\R^p$,
let $u_1^*\ge u_2^*\ge...\ge u_p^*\ge 0$
be a non-increasing rearrangement of $|u_1|,...,|u_p|$.
Then the norm $|\cdot|_*$ can be equivalently defined  as
\begin{equation*}
    |\vu|_* = \sum_{j=1}^p  \mu_j u_i^*,
    \qquad
    \vu=(u_1,...,u_p) \in\R^p.
\end{equation*}
Given a tuning parameter $A>0$, we define the SLOPE estimator $\hbeta$  
as a solution of the optimization problem
\begin{equation}
    \hbeta \in \argmin_{\vbeta\in\R^p}\left(
    \norm{ \vy-\design\vbeta}^2
    + 2 A |\vbeta|_*\right).
    \label{eq:def-slope}
\end{equation}
As $|\cdot|_*$ is a norm, it is a convex function 
so Proposition \ref{prop3} and Theorem \ref{cor3} apply.
For any $s\in \{1,...,p\}$ and any $c_0>0$, the Weighted Restricted Eigenvalue (WRE) constant $\vartheta(s,c_0)\ge 0$ is defined as follows:
\begin{equation*}
    \vartheta^2(s,c_0)
    \triangleq
    \min_{
        \vdelta\in\R^p: 
        \sum_{j=s+1}^p \mu_j \delta_j^* \le c_0 (\sum_{j=1}^s \mu_j^2)^{1/2} |\vdelta|_2
    }
    \frac{\norm{\design\vdelta}^2}{|\vdelta|_2^2}.
\end{equation*}
The $WRE(s,c_0)$ condition is said to hold if $\vartheta(s,c_0)>0$.

We refer the reader to  \cite{bellec2016slope} 
for a comparison of this RE-type constant with other restricted eigenvalue constants such as \eqref{eq:def-kappa}. A high level message is that the WRE condition is only slightly stronger than the RE condition.
It is also established in \cite{bellec2016slope} 
that a large class of random matrices $\design$ with independent and possibly anisotropic rows satisfies the condition $\vartheta(s,c_0) > 0$ with high probability
provided that $n > C s \log(p/s)$ for some absolute constant $C>0$.

For $j=1,...,p$, let $g_j = \frac{1}{\sqrt n}\ve_j \design^T \vxi$, where $\ve_j$ is the $j$th canonical basis vector in $\R^p$,
and let $g_1^*\ge g_2^*\ge...\ge g_p^*\ge 0$
be a non-increasing rearrangement of $|g_1|,...,|g_p|$.
Consider the event
\begin{equation}
\Omega_* \triangleq
    \cap_{j=1}^p
    \left\{
        g_j^*
        \le 4\sigma\sqrt{\log(2p/j)}
    \right\} .
    \label{eq:def-Omega-star}
\end{equation}
The next proposition establishes a deterministic result for the SLOPE estimator on the event \eqref{eq:def-Omega-star}.

\begin{proposition}
    \label{prop:deterministic-slope}
    On the event \eqref{eq:def-Omega-star}, the SLOPE estimator $\hbeta$ defined by \eqref{eq:def-slope}
    with $A\ge8$ satisfies 
    \begin{equation}
        \norm{
            \design\hbeta
            -
            \vf
        }^2
        \le
        \min_{s \in\{1,...,p\}} \left[
        \min_{\vbeta\in\R^p: |\vbeta|_0 \le s}
                   \norm{
                \design\vbeta
                -
                \vf
            }^2
            + \frac{9 A^2 \sigma^2 s \log(2ep/s)}{4 n \vartheta^2(s,3)}
        \right]
        \label{eq:soi-squared-slope}
    \end{equation}
    with the convention that $a/0=+\infty$ for any $a>0$.
\end{proposition}
\begin{proof}
    \smartqed
    Let $s\in\{1,...,p\}$ and $\vbeta$ be minimizers of the right hand side of
    \eqref{eq:soi-squared-slope} and let $\vdelta \triangleq \hbeta - \vbeta$. We will assume that $\vartheta(s,3)>0$ since otherwise the claim is trivial.
     From Lemma \ref{lem1} with $F(\vbeta)= A|\vbeta|_*$ we have    
     \begin{equation*}
        \norm{
            \design\hbeta
            -
            \vf
        }^2
        -
        \norm{
            \design\vbeta
            -
            \vf
        }^2
        \le 2\left( \tfrac 1 n \vxi^T\design \vdelta + A |\vbeta|_* - A|\hbeta|_*\right) 
        - \norm{\design\vdelta}^2
        \triangleq D.
    \end{equation*}
    On the event \eqref{eq:def-Omega-star},
    the right hand side of the previous display satisfies
    \begin{equation*}
       D
        \le
        2A \left[
            \frac 1 2|\vdelta|_*   
            + |\vbeta|_* - |\hbeta|_*
        \right] - \norm{\design\vdelta}^2.
    \end{equation*}
    By \cite[Lemma A.1]{bellec2016slope},
    $
    \frac 1 2|\vdelta|_*   
        + |\vbeta|_* - |\hbeta|_*
        \le
        \frac 3 2 (\sum_{j=1}^s \mu_j^2)^{1/2} |\vdelta|_2 - \frac 1 2 \sum_{j=s+1}^p \mu_j \delta_j^*
    $.

    If $3 (\sum_{j=1}^s \mu_j^2)^{1/2} |\vdelta|_2 \le  \sum_{j=s+1}^p \mu_j \delta_j^*$,
    then the claim follows trivially.
   If the reverse inequality holds, we have $|\vdelta|_2 \le \norm{\design\vdelta}/\vartheta(s,3)$.
    This implies
    \begin{equation*}
        3A  (\sum_{j=1}^s \mu_j^2)^{1/2}   |\vdelta|_2
        \le 
        \frac{9 A^2\sum_{j=1}^s \mu_j^2}{4\vartheta^2(s,3)} + \norm{\design\vdelta}^2
        \le
        \frac{9 A^2 \sigma^2s\log(2ep/s)}{4 n \vartheta^2(s,3)} + \norm{\design\vdelta}^2,
    \end{equation*}
    where for the last inequality we have used that, by Stirling's formula,
    $\log(1/s!) \le s\log(e/s)$ and thus $\sum_{j=1}^s \mu_j^2 \le \sigma^2 s\log(2ep/s)/n$. Combining the last three displays yields the result.
    \qed
\end{proof}

We now follow the same argument as in Sections \ref{s:lasso} and \ref{s:glasso}. In order to apply Theorem~\ref{cor3},
we need to find a ``weak bound" $R$ on the error $\norm{\design\hbeta - \vf}$, i.e., a bound valid
with probability at least $1/2$. If the event $\Omega_*$ holds with probability at least 1/2 then, due to Proposition~\ref{prop:deterministic-slope}, we can take as $R$ the square root of the right hand side of \eqref{eq:soi-squared-slope}. 
Since $\vxi\sim \mathcal N(\vzero, \sigma^2 I_{n\times n})$
and the diagonal elements of $\frac 1 n
\design^T \design$ are bounded by 1,
the random variables  $g_1,...,g_p$ 
are centered Gaussian with variance
    at most $\sigma^2$.
The following proposition from \cite{bellec2016slope} shows
that the event \eqref{eq:def-Omega-star} has probability at least $1/2$.

\begin{proposition}{\rm\cite{bellec2016slope}}
    \label{prop:proba-estimate-slope}
    If $g_1,...,g_p$ are centered Gaussian random variables with variance
    at most $\sigma^2$,
    then the event
    \eqref{eq:def-Omega-star}
    has probability at least $1/2$.
\end{proposition}
Proposition \ref{prop:proba-estimate-slope} cannot be substantially improved
without additional assumptions.
To see this, let $\eta\sim\mathcal N(0,1)$
and set $g_j = \sigma \eta $ for all $j=1,...,p$.
The random variables $g_1,...,g_p$ satisfy the assumption of Proposition \ref{prop:proba-estimate-slope}.
In this case, the event \eqref{eq:def-Omega-star} satisfies
$\mathbb P ( \Omega_*) = \mathbb P( |\eta| \le 4\sqrt{\log 2})$
so that $\mathbb P(\Omega_*)$ is an absolute constant.
Thus, without additional assumptions
on the random variables $g_1,...,g_p$,
there is no hope to prove a lower bound better than $\mathbb P(\Omega_*) \ge c$ for some fixed numerical constant $c\in (0,1)$ independent of $p$.

By combining Propositions \ref{prop:deterministic-slope}
and \ref{prop:proba-estimate-slope},
we obtain that the oracle bound \eqref{eq:soi-squared-slope}
holds with probability at least $1/2$.
At first sight, this result is uninformative as it cannot 
even imply the consistency, i.e., the convergence of the error $\norm{\design\hbeta
- \vf}$ to 0 in probability.
But the SLOPE estimator is a convex regularized estimator
and the argument of Section \ref{s:machinery} yields
that a risk bound with probability $1/2$ is in fact
very informative:
Theorem \ref{cor3} immediately
implies the following oracle inequality for any confidence level $1-\delta$ as well as an oracle inequality in expectation.

\begin{theorem}
    \label{thm:slope}
    Assume that the diagonal elements of the matrix $\frac 1 n
    \design^T \design$ are not greater than 1.
    Then for all $\delta\in(0,1)$, the SLOPE
    estimator $\hbeta$ defined by \eqref{eq:def-slope} with tuning parameter $A\ge 8$
    satisfies
    \begin{equation*}
        \norm{
            \design\hbeta
            -
            \vf
        }
        \le
        \min_{s \in\{1,...,p\}}
        \left[
        \min_{\vbeta\in\R^p: |\vbeta|_0 \le s}
            \norm{
                \design\vbeta
                -
                \vf
            }
            + \frac{3\sigma A }{2 \vartheta(s,3)}
              \sqrt{\frac{s \log(2ep/s)}{n}}
        \right]
            + \frac{\sigma \Phi^{-1}(1-\delta)}{\sqrt{n}}
        \label{eq:all-delta-slope}
    \end{equation*}
    with probability at least $1-\delta$. Furthermore,
    \begin{equation*}
      \E  \norm{
            \design\hbeta
            -
            \vf
        }
        \le
        \min_{s \in\{1,...,p\}}
        \left[
        \min_{\vbeta\in\R^p: |\vbeta|_0 \le s}
            \norm{
                \design\vbeta
                -
                \vf
            }
            + \frac{3\sigma A }{2 \vartheta(s,3)}
              \sqrt{\frac{s \log(2ep/s)}{n}}
        \right]
               + \frac{\sigma}{ \sqrt{2 \pi n}}.
        \label{eq:}
         \end{equation*}
\end{theorem}
  The proof is similar to that of Theorem \ref{thm:lasso},
    and thus it is omitted. Remarks analogous to the discussion after Theorem \ref{thm:lasso} apply here as well.

\section{Generalizations and extensions}\label{s:Generalizations}

The list of applications of Theorem~\ref{cor3} considered in the previous sections can be further extended. For instance, the same techniques can be applied when, instead of prediction by $\design \vbeta$ for $\vf$, one uses a trace regression prediction. In this case, the estimator $\hbeta\in\R^{m_1\times m_2}$ is a matrix satisfying
\begin{equation}
    \hbeta(\vy) \in \argmin_{\vbeta\in\R^{m_1\times m_2}} \left(\frac1n \sum_{i=1}^n (y_i - {\rm trace}(X_i^T\vbeta))^2 + 2 F(\vbeta)\right),
    \label{eq:def-hbeta-trace-reg}
\end{equation}
where $X_1,...,X_n$ are given deterministic matrices in  $\R^{m_1\times m_2}$, $F:\R^{m_1\times m_2}\to \R$ is a convex penalty. A popular example of $F(\vbeta)$ in this context is the nuclear norm of $\vbeta$.  The methods of this paper apply for such an estimator as well, and we obtain analogous bounds. Indeed, \eqref{eq:def-hbeta-trace-reg} can be rephrased as \eqref{eq:def-hbeta-pen} by vectorizing $\vbeta$ and defining a new matrix~$\design$. Thus, Theorem~\ref{cor3} can be applied. Next, note that the examples of application of Theorem~\ref{cor3} considered above required only two ingredients: a deterministic oracle inequality and a weak bound on the probability of the corresponding random event.  The deterministic bound is obtained here quite analogously to the previous sections or following the same lines as in \cite{koltchinskii2011nuclear} or in \cite[Corollary 12.8]{vdgeer}. A bound on the probability of the random event can be also borrowed from \cite{koltchinskii2011nuclear}.  We omit further details. 

Finally, we observe that Proposistion~\ref{prop3} and Theorem~\ref{cor3} generalize to Hilbert space setting. 
Let $H, H'$ be two Hilbert spaces and $\design : H'\to H$ a bounded linear operator. If $H$ is equipped with a norm $\|\cdot\|_H$, and $F:H'\to [0,+\infty]$ is a proper convex function, consider for any $\vy\in H$ a solution
\begin{equation}
    \hbeta(\vy) \in \argmin_{\vbeta\in H'} \left(\|\vy - \design \vbeta\|_H^2 + 2 F(\vbeta)\right).
    \label{eq:def-hbeta-hilbert}
\end{equation}
\begin{proposition}
Under the above assumptions, any solution  $\hbeta(\vy)$ of \eqref{eq:def-hbeta-hilbert} satisfies $\|\design (\hbeta(\vy)-\hbeta(\vy'))\|_H\le \|\vy-\vy'\|_H$. 
\end{proposition}
The proof of this proposition is completely analogous to that of Proposistion~\ref{prop3}. It suffices to note that the properties of convex functions used in the proof of Proposistion~\ref{prop3} are valid when these functions are defined on a Hilbert space, cf. \cite{peypouquet2015convex}.
This and the fact that the Gaussian concentration property extends to Hilbert space valued Gaussian variables \cite[Theorem 6.2]{lifshits} immediately imply a Hilbert space analog of Theorem~\ref{cor3}.


\bibliographystyle{plainnat}
\bibliography{konakov}

\end{document}